\documentclass[letterpaper, 10 pt, conference]{ieeeconf}

\IEEEoverridecommandlockouts
\usepackage{graphicx} 
\usepackage{epsfig} 
\usepackage{mathptmx} 
\usepackage{amsfonts}
\usepackage{amsmath} 
\usepackage{amssymb}  
\usepackage{stdmath}  
\usepackage{subfigure}
\usepackage[lined,boxed,commentsnumbered,vlined,ruled]{algorithm2e}
\usepackage{verbatim}

\DeclareFontFamily{OT1}{pzc}{}
\DeclareFontShape{OT1}{pzc}{m}{it}{<-> s * [1.10] pzcmi7t}{}
\DeclareMathAlphabet{\mathpzc}{OT1}{pzc}{m}{it}
\newtheorem{mydef}{Definition}

\newtheorem{mylem}{Lemma}

\newtheorem{myass}{Assumption}

\title{\LARGE \bf
Constructing Piecewise-Polynomial Lyapunov Functions for Local Stability of Nonlinear Systems Using Handelman's Theorem
}

\author{Reza Kamyar, Chaitanya Murti and Matthew M. Peet 
\thanks{R. Kamyar, C. Murti and M. M. Peet are with the Cybernetic Systems and Controls Lab (CSCL) at Arizona State University, Tempe, AZ, 85281 USA, {rkamyar@asu.edu \tt\small}, {cmurti@hawk.iit.edu \tt\small}, {mpeet@asu.edu \tt\small }}%
\thanks{This work was made possible by the National Science Foundation under grants \# CMMI-1301660. }%
}

\begin{document}

\maketitle
\thispagestyle{empty}
\pagestyle{empty}

\begin{abstract}
In this paper, we propose a new convex approach to stability analysis of nonlinear systems with polynomial vector fields. First, we consider an arbitrary convex polytope that contains the equilibrium in its interior. Then, we decompose the polytope into several convex sub-polytopes with a common vertex at the equilibrium. Then, by using Handelman's theorem, we derive a new set of affine feasibility conditions -solvable by linear programming- on each sub-polytope. Any solution to this feasibility problem yields a piecewise polynomial Lyapunov function on the entire polytope.
This is the first result which utilizes Handelman's theorem and decomposition to construct piecewise polynomial Lyapunov functions on arbitrary polytopes.
 In a computational complexity analysis, we show that for large number of states and large degrees of the Lyapunov function, the complexity of the proposed feasibility problem is less than the complexity of certain semi-definite programs associated with alternative methods based on Sum-of-Squares or Polya's theorem. Using different types of convex polytopes, we assess the accuracy of the algorithm in estimating the region of attraction of the equilibrium point of the reverse-time Van Der Pol oscillator.

 \end{abstract}

\section{Introduction}

One approach to stability analysis of nonlinear systems is the search for a decreasing Lyapunov function.  For those systems with polynomial vector fields, searching for polynomial Lyapunov functions has been shown to be necessary and sufficient for stability on any bounded set~\cite{peet2012converse}. However, searching for a polynomial Lyapunov function which proves local stability requires enforcing positivity on a neighborhood of the equilibrium. Unfortunately, while we do have necessary and sufficient conditions for positivity of a polynomial (e.g. Tarski-Seidenberg~\cite{Tarski}, Artin~\cite{Artin}), it has been shown that the general problem of determining whether a polynomial is positive is NP-hard~\cite{blum}.



The most well-known approach to determining positivity of a polynomial is to search for a representation as the sum and quotient of squared polynomials~\cite{powers2011positive}. Such a representation is necessary and sufficient for a polynomial to be positive semidefinite. If we leave off the quotient, the search for a Sum-of-Squares (SOS) is a common sufficient condition for positivity of a polynomial.
The advantage of the SOS approach is that verifying the existence of an SOS representation is a  semidefinite programming problem~\cite{papachristodoulou2013sostools}. This approach was first articulated in~\cite{parrilo}. SOS programming has been used extensively in stability analysis and control including stability analysis of nonlinear systems~\cite{tan2008stability}, robust stability analysis of switched and hybrid systems~\cite{prajna_hybrid}, and stability analysis of time-delay systems~\cite{papachristodoulou2009analysis}.

In addition to the SOS representation of positive polynomials, there exist alternative representation theorems for polynomials which are not globally positive. For example, Polya's Theorem~\cite{polya_book} states that every strictly positive homogeneous polynomial on the positive orthant can be represented as a sum of even-powered monomials with positive coefficients. Multiple variants of Polya's theorem have been proposed, e.g., extensions to the multi-simplex or hypercube~\cite{peres2007parameter,kamyar2012hypercube}, an extension to polynomials with zeros on the boundary of the simplex~\cite{castle2011polya} and an extension to the entire real domain~\cite{de1996effective}.

The downside to the use of SOS (with Positivstellensatz multipliers) or Polya's algorithm for stability analysis of nonlinear systems with many states is computational complexity. Specifically, these methods require us to set up and solve large SDPs. For example, using the SOS algorithm to construct a degree $6$ Lyapunov function on the hypercube for a system with $10$ states implies an SDP with $ \sim 10^8$ variables and  $\sim 10^5$ constraints. Although Polya's algorithm implies similar complexity to SOS, the SDPs associated with Polya's algorithm possess a block-diagonal structure. This has allowed some work on parallel computing approaches such as can be found in~\cite{kamyar2012solving,kamyar2013nonlinear} for robust stability and nonlinear stability, respectively. However, although Polya's algorithm has been generalized to positivity over simplices and hypercubes; as yet no generalization exists for arbitrary convex polytopes. Therefore, in this paper, we look at Handelman's theorem~\cite{handelman}. Specifically, given an arbitrary convex polytope, Handelman's theorem provides a parameterization of all polynomials that are positive on the given polytope.

Some preliminary work on the use of Handelman's theorem and interval evaluation for Lyapunov functions on the hypercube has been suggested in~\cite{2013handelmanlyapunov} and has also been applied to robust stability of positive linear systems in~\cite{briat}.
In this paper, we consider a new approach to the use of Handelman's theorem for computing regions of attraction of stable equilibria by constructing piecewise-polynomial Lyapunov functions on arbitrary convex polytopes. Specifically, we decompose a given convex polytope into a set of convex sub-polytopes that share a common vertex at the origin. Then, on each sub-polytope, we convert Handelman's conditions to linear programming constraints. Additional constraints are then proposed which ensure continuity of the Lyapunov function. We then show the resulting algorithm has polynomial complexity in the number of states and compare this complexity with algorithms based on SOS and Polya's theorem. Finally, we evaluate the accuracy of our algorithm by numerically approximating the domain of attraction of the reverse-time Van Der Pol oscillator.

\section{Definitions and Notation}
\label{Sec:Notation}

In this section, we define convex polytopes, facets of polytopes, decompositions and Handelman bases. \vspace{0.05in}

\begin{mydef} (Convex Polytope)
\label{df:polytope1}
Given the set of vertices $P := \{ p_i \in \mathbb{R}^n, i=1, \cdots, K \}$, define the \emph{convex polytope} $\Gamma_P$ as \vspace{-0.07in}
\[
\Gamma_P := \{ x \in \mathbb{R}^n: x = \sum_{i=1}^K \mu_i p_i : \mu_i \in [0,1]
\text{ and } \sum_{i=1}^K \mu_i = 1 \}.
\]
\end{mydef}
Every convex polytope can be represented as
\[
\Gamma := \{ x \in \mathbb{R}^n: w_i^T x+u_i \geq 0, i=1, \cdots, K\},
\]
for some $w_i \in \mathbb{R}^n, u_i \in \mathbb{R}, i=1, \cdots, K$. Throughout the paper, every polytope that we use contains the origin. \vspace{0.05in}

\begin{mydef}
Given a bounded polytope of the form $\Gamma:= \{ x \in \mathbb{R}^n: w_i^T x + u_i \geq 0, i=1, \cdots, K \}$, we call
\begin{align*}
\zeta^i(\Gamma)&:=\{ x \in \mathbb{R}^n: w_i^Tx+u_i = 0 \text{ and } w_j^Tx+u_j \geq 0 \\
&\qquad\qquad\qquad\quad\quad\quad\;\;\;\; \text{ for } j \in \{1, \cdots,K \} \}
\end{align*}
the $i-$th facet of $\Gamma$ if $\zeta^i(\Gamma) \neq \emptyset$.
\end{mydef}

\vspace{0.05in}

\begin{mydef}($D-$decomposition)
\label{df:decomposition}
Given a bounded polytope of the form $\Gamma:= \{ x \in \mathbb{R}^n: w_i^T x + u_i \geq 0, i=1, \cdots, K \}$, we call $D_{\Gamma}:=\{D_i\}_{i=1,\cdots,L}$ a $D-$\textit{decomposition} of $\Gamma$ if
\begin{equation*}
D_i := \{ x \in \mathbb{R}^n: h^T_{i,j} x + g_{i,j} \geq 0, j=1, \cdots,m_i \}
\label{eq:decomposition}
\end{equation*}
$\text{ for some } h_{i,j} \in \mathbb{R}^n, \, g_{i,j} \in \mathbb{R}$, such that $\cup_{i=1}^L D_i = \Gamma,\, \cap_{i=1}^L D_i = \{0\}$ and $\text{int}(D_i) \cap \text{int}(D_j) = \emptyset$.
\end{mydef} \vspace{0.05in}

\begin{mydef} (The Handelman basis associated with a polytope)
\label{df:Handelman_basis}
Given a polytope of the form
\[\Gamma:=\{x\in\mathbb{R}^n: w_i^Tx+u_i \geq 0,\;i = 1,\cdots,K\},\]
we define the set of \emph{Handelman bases}, indexed by
\begin{equation}\label{eq:E_d}
\alpha\in E_{d,K}:=\{\alpha\in\mathbb{N}^K:|\alpha|_1\leq d\}\end{equation} \vspace{-0.07in}
as \vspace{-0.07in}
\[\Theta_d(\Gamma):=\{\rho_\alpha(x): \rho_\alpha(x) = \prod_{i=1}^K(w_i^Tx+u_i)^{\alpha_i},\;\alpha\in E_{d,K} \}.\]

\end{mydef}

\begin{mydef}(Restriction of a polynomial to a facet)\label{df:rstr_poly}
Given a polytope of the form $\Gamma:=\{x\in\mathbb{R}^n:w_i^Tx+u_i,\;i=1,\cdots, K\}$, and a polynomial $P(x)$ of the form
\[
P(x) = \sum_{\alpha\in E_{d,K}}b_{\alpha}\prod_{i=1}^K(w_i^Tx+u_i)^{\alpha_i},
\]
define the \emph{restriction of $P(x)$ to the $k$-th facet of $\Gamma$} as the function \vspace{-0.07in}
\[
P|_{k}(x) := \sum_{\substack{\alpha\in E_d:\alpha_k=0}}b_{\alpha}\prod_{i=1}^K(w_i^Tx+u_i)^{\alpha_i}. \vspace{-0.07in}
\]
\end{mydef}
\vspace{.2in}

We will use the maps defined below in future sections.
\begin{mydef}
\label{df:maps}
 Given $w_i, h_{i,j} \in \mathbb{R}^n$ and $u_i, g_{i,j} \in \mathbb{R}$, let $\Gamma$ be a convex polytope as defined in Definition~\ref{df:polytope1} with $D-$decomposition $D_{\Gamma}:=\{D_i\}_{i=1,\cdots,L}$ as defined in Definition~\ref{df:decomposition}, and let $\lambda^{(k)},\, k=1, \cdots,B$ be the elements of $E_{d,n}$, as defined in \eqref{eq:E_d}, for some $d,n,\in \mathbb{N}$.  For any $\lambda^{(k)} \in E_{d,n}$, let $p_{\{\lambda^{(k)},\alpha,i\}}$ be the coefficient of $b_{i,\alpha} x^{\lambda^{(k)}}$ in \vspace{-0.07in}
\begin{equation}
P_i(x) := \sum_{\alpha \in E_{d,m_i}} b_{i,\alpha} \prod_{j=1}^{m_i} (h^T_{i,j}x+g_{i,j})^{\alpha_j}. \vspace{-0.07in}
\label{eq:V_i}
\end{equation}
Let $N_i$ be the cardinality of $E_{d,m_i}$, and denote by $b_i \in\mathbb{R}^{N_i}$ the vector of all coefficients $b_{i,\alpha}$ .\\
Define $F_{i}: \mathbb{R}^{N_i} \times \mathbb{N} \rightarrow \mathbb{R}^{B}$ as \vspace{-0.07in}
\begin{equation}\label{eq:Fdef}
F_{i}(b_{i},d) := \left[
 \sum_{\alpha \in E_{d,m_i}}  p_{\{\lambda^{(1)},\alpha,i\}} b_{i,\alpha} , \; \cdots \; ,\hspace{-.2cm}  \sum_{\alpha \in E_{d,m_i}}  p_{\{\lambda^{(
B)},\alpha,i\}} b_{i,\alpha} \right]^T
\end{equation}
for  $i=1, \cdots,L$. In other words, $F_i(b_i,d)$ is the vector of the coefficients of $P_i(x)$ after expansion.\\
\noindent Define $H_{i}: \mathbb{R}^{N_i} \times \mathbb{N} \rightarrow \mathbb{R}^{Q} $ as \vspace{-0.07in}
\begin{equation}\label{eq:Hdef}
H_{i}(b_i,d) := \left[
 \sum_{\alpha \in E_{d,m_i}}\hspace{-.1cm}  p_{\{\delta^{(1)},\alpha,i\}} b_{i,\alpha } \; , \; \cdots \; , \hspace{-.2cm}  \sum_{\alpha \in E_{d,m_i}} \hspace{-.1cm} p_{\{\delta^{(
Q)},\alpha,i\}} b_{i,\alpha} \right]^T
\end{equation}
for $i = 1, \cdots,L$, where we have denoted the elements of $\{ \delta \in \mathbb{N}^n: \delta = 2e_j \text{ for } j=1, \cdots,n \}$ by $\delta^{(k)}, k=1, \cdots, Q$, where $e_j$ are the canonical basis for $\mathbb{N}^n$. In other words, $H_i(b_i,d)$ is the vector of coefficients of square terms of $P_i(x)$ after expansion.\\
\noindent Define $J_{i}: \mathbb{R}^{N_i} \times \mathbb{N}\times \{1,\cdots,m_i\} \rightarrow \mathbb{R}^{B} $ as \vspace{-0.07in}
\begin{equation}\label{eq:Jref}
J_{i}(b_{i},d,k) := \left[
 \sum_{\substack{\alpha \in E_{d,m_i}\\\alpha_k=0 }}  \hspace{-.3cm} p_{\{\lambda^{(1)},\alpha,i\}} b_{i,\alpha} \; \cdots \; ,  \sum_{\substack{\alpha \in E_{d,m_i}\\\alpha_k=0} }\hspace{-.3cm}  p_{\{\lambda^{(
B)},\alpha,i\}} b_{i,\alpha} \right]^T
\end{equation}
for $i = 1, \cdots,L$. In other words, $J_i(b_i,d,k)$ is the vector of coefficients of $P_i|_k(x)$ after expansion.\\
\noindent Given a polynomial vector field $f(x)$ of degree $d_f$, define $G_{i}: \mathbb{R}^{N_i} \times \mathbb{N} \rightarrow \mathbb{R}^{Z} $ as \vspace{-0.07in}
\begin{equation}\label{eq:Gdef} 
G_{i}(b_{i},d) := \left[
 \sum_{\alpha \in E_{d,m_i}}  s_{\{\eta^{(1)},\alpha,i\}} b_{i,\alpha} \; , \; \cdots \; , \hspace{-.2cm} \sum_{\alpha \in E_{d,m_i}}  s_{\{\eta^{(
P)},\alpha,i\}} b_{i,\alpha} \right]^T
\end{equation}
for $i=1, \cdots,L$,  and where we have denoted the elements of $E_{d+d_f-1,n}$ by $\eta^{(k)}$, $k=1,\cdots,Z$. For any $\eta^{(k)} \in E_{d+d_f-1,n}$, we define $s_{\{ \eta^{(k)}, \alpha, i \}}$ as the coefficient of $b_{i,\alpha} x^{\eta^{(k)}}$ in $\langle \nabla P_i(x),f(x) \rangle$, where $P_i(x)$ is defined in~\eqref{eq:V_i}. In other words, $G_{i}(b_{i},d)$ is the vector of coefficients of $\langle \nabla P_i(x), f(x) \rangle$.\\
\noindent Define $R_i(b_{i},d): \mathbb{R}^{N_i} \times \mathbb{N} \rightarrow \mathbb{R}^C$ as
\begin{equation}\label{eq:Rdef}
R_i(b_{i},d):=
\left[ b_{i,\beta^{(1)}} \; , \; \cdots \; , \; b_{i,\beta^{(C)}} \right]^T,
\end{equation}
for $i = 1, \cdots,L$, where we have denoted the elements of
\[
S_{d,m_i}:=\{ \beta \in E_{d,m_i} :  \beta_j = 0 \text{ for } j \in \{ j \in \mathbb{N}: g_{i,j} = 0 \} \}
\]
 by $\beta^{(k)}$, $k=1,\cdots,C$. Consider $P_i$  in the Handelman basis $\Theta_d(\Gamma)$. Then, $R_i(b_i,d)$ is the vector of coefficients of monomials of $P_i$ which are nonzero at the origin.\\
 It can be shown that the maps $F_i,\;H_i,\;J_i,\;G_i\text{ and }R_i$ are affine in $b_i$.
\end{mydef}

\begin{mydef}(Upper Dini Derivative)
\label{df:Dini} Let $f:\mathbb{R}^n\rightarrow \mathbb{R}^n$ be a continuous map. Then, define the upper Dini derivative of a function $V:\mathbb{R}^n\rightarrow\mathbb{R}$ in the direction $f(x)$
as
\[D^+(V(x),f(x)) = \limsup_{h\rightarrow 0^+}\frac{V(x+hf(x))-V(x)}{h}.\]
\end{mydef}
It can be shown that for a continuously differentiable $V(x)$,
\[D^+(V(x),f(x)) = \langle \nabla V(x),f(x)\rangle.
\]

\section{Background and Problem Statement}
\label{sec:Background}

We address the problem of local stability of nonlinear systems of the form \vspace{-0.1in}
\begin{equation}
\dot{x}(t) = f(x(t)),
\label{eq:system}
\end{equation}
about the zero equilibrium, where $f:\mathbb{R}^n \rightarrow \mathbb{R}^n$. We use the following Lyapunov stability condition.

\begin{thm}
\label{thm:lyap}
For any $\Omega \subset \mathbb{R}^n$ with $0 \in \Omega$, suppose there exists
a continuous function $V: \mathbb{R}^n \rightarrow \mathbb{R} $ and continuous positive definite functions $W_1, W_2, W_3$,
\begin{align*}
& W_1(x) \leq V(x) \leq W_2 (x) \text{ for } x \in  \Omega  \text{ and} \\
&  D^+(V(x),f(x)) \leq -W_3(x)  \text{ for } x \in \Omega,
\end{align*}
then System~\eqref{eq:system} is asymptotically stable on $\{ x: \{ y: V(y) \leq V(x) \} \subset \Omega \}$.
\end{thm}
In this paper, we construct piecewise-polynomial Lyapunov functions which may not have classical derivatives. As such, we use Dini derivatives which are known to exist for  piecewise-polynomial functions. \vspace{0.05in}

\textit{\textbf{Problem statement}:}\label{pb:prob1}
Given the vertices $p_i \in \mathbb{R}^n, i=1, \cdots, K$, we would like to find 
the largest positive $s$ such that there exists a polynomial $V(x)$ where $V(x)$ satisfies the conditions of Theorem~\ref{thm:lyap} on the convex polytope
$\{ x \in \mathbb{R}^n : x = \sum_{i=1}^K \mu_i p_i: \mu_i \in [0,s] \text{ and } \sum_{i=1}^K \mu_i = s \}$.


Given a convex polytope, the following result~\cite{handelman} parameterizes the set of polynomials which are positive on that polytope using the positive orthant.

\begin{thm} (Handelman's Theorem)
\label{thm:Handelman} Given $w_i \in \mathbb{R}^n, u_i \in \mathbb{R}, i=1,\cdots,K$, let $\Gamma$ be a convex polytope as defined in definition~\ref{df:polytope1}. If polynomial $P(x) > 0$ for all $x \in \Gamma$, then there exist $b_\alpha \geq 0$, $\alpha \in \mathbb{N}^K$ such that for some $d \in \mathbb{N}$, \vspace{-0.05in}
\[
P(x) := \sum_{\alpha \in E_{d,K}} b_{\alpha} \prod_{ji=1}^{K} (w^T_{i}x+u_{i})^{\alpha_i}.
\]
\end{thm}


Given a D-decomposition $D_\Gamma := \{ D_i \}_{i=1, \cdots,L}$ of the form
\[D_i :=\{ x \in \mathbb{R}^n: h_{i,j}^T x + g_{i,j} \geq 0, j=1, \cdots, m_i \}
\]
of some polytope $\Gamma$, we parameterize a cone of piecewise-polynomial Lyapunov functions which are positive on $\Gamma$ as
\begin{align*}
V(x) = V_i(x) := \sum_{\alpha \in E_{d,m_i}} b_{i,\alpha} \prod_{j=1}^{m_i} (h^T_{i,j}x+g_{i,j})^{\alpha_j}, \\
 \text{ for } x \in D_i \text{ and } i=1,\cdots,L.
\end{align*}
We will use a similar parameterization of piecewise-polynomials which are negative on $\Gamma$ in order to enforce negativity of the derivative of the Lyapunov function. We will also use linear equality constraints to enforce continuity of the Lyapunov function.


\section{Problem setup}
\label{sec:setup}
We first present some lemmas necessary for the proof of our main result. The following lemma provides a sufficient condition for a polynomial represented in the Handelman basis to vanish at the origin ($V(0)=0$).

\begin{mylem}
\label{Zero Origin}
 Let $D_{\Gamma} := \{ D_i \}_{i=1,\cdots,L}$ be a D-decomposition of a convex polytope $\Gamma$, where
\[
D_i := \{ x \in \mathbb{R}^n: h_{i,j}^T x + g_{i,j} \geq 0, j=1, \cdots, m_i \}.
\]
For each $i\in \{1\cdots,L\}$, let \vspace{-0.05in}
\[
P_i(x) := \sum_{\alpha \in E_{d,m_i}} b_{i,\alpha} \prod_{j=1}^{m_i} (h^T_{i,j}x+g_{i,j})^{\alpha_j}, \vspace{-0.05in}
\]
 $N_i$ be the cardinality $E_{d,m_i}$ as defined in \eqref{eq:E_d}, and let $b_i\in\mathbb{R}^{N_i}$ be the vector of the coefficients $b_{i,\alpha}$ . Consider $R_i:\mathbb{R}^{N_i}\times \mathbb{N}\rightarrow \mathbb{R}^C$ as defined in \eqref{eq:Rdef}. If $R_i(b_i,d)=\mathbf{0}$, then $P_i(x) = 0$ for all $i\in \{1\cdots,L\}$.
\end{mylem}
\begin{proof}
We can write
\[
P_i(x) = \hspace{-.35cm}\sum_{\alpha \in E_{d,m_i}\backslash S_{d,m_i}} \hspace{-0.55cm} b_{i,\alpha}\prod_{j=1}^{m_i} (h_{i,j}^Tx+g_{i,j})^{\alpha_i} + \hspace{-0.25cm}\sum_{\alpha \in S_{d,m_i}} \hspace{-0.25cm}  b_{i,\alpha} \prod_{j=1}^{m_i} (h_{i,j}^Tx+g_{i,j})^{\alpha_i} ,
\]
where
\[S_{d,m_i} :=\{ \alpha \in E_{d,m_i} :  \alpha_j = 0 \text{ for } j \in \{ j \in \mathbb{N}: g_{i,j} = 0 \} \}.\]
By the definitions of $E_{d,m_i}$ and $S_{d,m_i}$, we know that for each $\alpha \in \ E_{d,m_i}\backslash S_{d,m_i}$ for $i\in\{1,\cdots,L\}$, there exists at least one $j\in \{1,\cdots,m_i\}$ such that $g_{i,j} = 0$ and $\alpha_k>0$. Thus, at $x=0$,
\[
\sum_{\alpha \in E_{d,m_i}\backslash S_{d,m_i}} \hspace{-.4cm} b_{i,\alpha}\prod_{j=1}^{m_i} (h_{i,j}^Tx+g_{i,j})^{\alpha_i} = 0\;\;\;\text{for all }i\in\{1,\cdots,L\}.
\]
Recall the definition of the map $R_i$ from \eqref{eq:Rdef}. Since $R_i(b_i,d)=\mathbf{0}$ for each $i\in\{1,\cdots,L\}$, it follows from that $b_{i,\alpha} = 0$ for each  $\alpha \in S_{d,m_i}$ and $i\in\{1,\cdots,L\}$. Thus,
\[
\sum_{\alpha \in S_{d,m_i}} \hspace{-0.05in}  b_{i,\alpha} \prod_{j=1}^{m_i} (h_{i,j}^Tx+g_{i,j})^{\alpha_i} = 0\;\;\;\text{for all }i\in\{1,\cdots,L\}.
\]
Thus, $P_i(0) = 0$ for all $i\in\{1,\cdots,L\}$.
\end{proof}

This Lemma provides a condition which ensures that a piecewise-polynomial function on a D-decomposition is continuous.
\begin{mylem}
\label{Continuity Lemma}
Let $D_{\Gamma} := \{ D_i \}_{i=1,\cdots,L}$ be a D-decomposition of a polytope $\Gamma$, where
\[
D_i := \{ x \in \mathbb{R}^n: h_{i,j}^T x + g_{i,j} \geq 0, j=1, \cdots, m_i \}.
\]
For each $i\in \{1\cdots,L\}$, let \vspace{-0.1in}
\[P_i(x) := \sum_{\alpha \in E_{d,m_i}} b_{i,\alpha} \prod_{j=1}^{m_i} (h^T_{i,j}x+g_{i,j})^{\alpha_j},\]
 $N_i$ be the cardinality of $E_{d,m_i}$ as defined in \eqref{eq:E_d}, and let $b_i\in \mathbb{R}^{N_i}$ be the vector of the coefficients $b_{i,\alpha}$. Given $i,j \in \{1, \cdots,L \},i \neq j$, let \vspace{-0.1in}
\begin{align}
 \nonumber\Lambda_{i,j}(D_\Gamma) \hspace{-.05cm}&:= \hspace{-.05cm}\left\{ k,l \in \mathbb{N} : k \in \{ 1, \cdots, m_i \}, l \in \{ 1, \cdots, m_j \}\hspace{-.05cm}:\hspace{-.05cm} \right.\\
 &\qquad\qquad\qquad \zeta^k(D_i) \neq \emptyset\text{ and } \left. \zeta^k(D_i) = \zeta^l (D_j) \right\}.\label{eq:lambdaij}\end{align}
  Consider $J_{i}:\mathbb{R}^{N_i}\times \mathbb{N}\times\{1\cdots,m_i\}\rightarrow \mathbb{R}^B$ as defined in \eqref{eq:Jref}. If \vspace{-0.1in}
\[
J_{i}(b_i,d,k) = J_{j}(b_j,d,l) \vspace{-0.1in}
\]
for all $i,j \in \{1, \cdots,L \},\;i\neq j$ and $k,l\in\Lambda_{i,j}(D_\Gamma)$, then the piecewise-polynomial function  \vspace{-0.05in}
\[
P(x)=P_i(x),\quad \text{ for } x\in D_i,\;i=1,\cdots,L   
\]
is continuous for all $x\in \Gamma$.
\end{mylem}

\begin{proof}
  From \eqref{eq:Jref},  $J_{i}(b_i,d,k)$ is the vector of coefficients of $P_i |_k(x)$ after expansion. Therefore, if $J_{i}(b_i,d,k) = J_{j}(b_j,d,l)$ for all $i,j\in \{1,\cdots,L\},\;i\neq j\text{ and }(k,l)\in\Lambda_{i,j}(D_\Gamma) $, then \begin{align}\nonumber P_i |_k(x) = &P_j |_l(x)\text{ for all }i,j\in \{1,\cdots,L\},\;i\neq j\text{ and }\\\label{eq:PiPj}&\qquad\qquad\qquad\qquad\qquad(k,l)\in\Lambda_{i,j}(D_\Gamma).
\end{align}
On the other hand, from definition \ref{df:rstr_poly}, it follows that for any $i\in\{1,\cdots ,L\}$ and $k\in\{1,\cdots,m_i\}$, \vspace{-0.05in} \begin{align}\label{eq:restr_equal}
P_i|_k(x) = P_i(x)\text{ for all }x\in\zeta^k(D_i). \vspace{-0.05in}
\end{align}
Furthermore, from the definition of $\Lambda_{i,j}(D_\Gamma)$, we know that
 \begin{equation}\label{eq:interface}
\zeta^k(D_i)= \zeta^l(D_j)\subset D_i\cap D_j\end{equation} for any $i,j\in\{1\cdots,L\},\;i\neq j$ and any $(k,l)\in\Lambda_{i,j}(D_\Gamma)$.
 Thus, from \eqref{eq:PiPj}, \eqref{eq:restr_equal} and \eqref{eq:interface}, it follows that for any $i,j\in\{1,\cdots,L\},\;i\neq j$, we have $P_i(x)=P_j(x)$ for all $x\in D_i\cap D_j$ . Since for each $i\in\{1,\cdots,L\}$, $P_i(x)$ is continuous on $D_i$ and for any $i,j\in\{1\cdots,L\},\;i\neq j$, $P_i(x) = P_j(x)$ for all $x\in D_i\cap D_j$, we conclude that the piecewise polynomial function
 \[P(x) = P_i(x)\quad x\in D_i,i=1,\cdots,L\]
 is continuous for all $x\in \Gamma$.
\end{proof}
\begin{thm}(Main Result)
\label{thm:main_theorem}
Let $d_f$ be the degree of the polynomial vector field $f(x)$ of System~\eqref{eq:system}. Given $w_i,\;h_{i,j}\in\mathbb{R}^n$ and $u_i,\;g_{i,j}\in\mathbb{R}$, define the polytope
\[
\Gamma := \{ x \in \mathbb{R}^n: w_i^T x+u_i \geq 0, i=1, \cdots, K\},
\]
with D-decomposition $D_{\Gamma} := \{ D_i \}_{i=1,\cdots,L}$, where
\[
D_i := \{ x \in \mathbb{R}^n: h_{i,j}^T x + g_{i,j} \geq 0, j=1, \cdots, m_i \}.
\]
 Let $N_i$ be the cardinality of $E_{d,m_i}$, as defined in \eqref{eq:E_d} and let $M_i$ be the cardinality of $E_{d+d_f-1,m_i}$. Consider the maps $R_i$, $H_i$, $F_i$, $G_i$, and $J_i$ as defined in definition \ref{df:maps}, and $\Lambda_{i,j}(D_\Gamma)$ as defined in \eqref{eq:lambdaij} for $i,j\in\{1,\cdots,L\}$. If there exists $d\in\mathbb{N}$ such that $\max \gamma$ in the linear program (LP),
\begin{align}
&\max_{\gamma \in \mathbb{R} , b_i \in \mathbb{R}^{N_i}, c_i \in \mathbb{R}^{M_i}} \;\;\; \gamma  \nonumber  \\
\nonumber&  \text{subject to } \;\;
\\& b_{i} \geq \mathbf{0}   &&\hspace{-.3cm} \text{ for } i=1, \cdots, L \nonumber \\
&  c_{i} \leq \mathbf{0}           && \hspace{-.3cm}\text{ for } i=1, \cdots, L \nonumber  \\
&  R_i(b_{i},d) = \mathbf{0}       && \hspace{-.3cm}\text{ for } i=1, \cdots, L        \nonumber \\
&  H_i(b_i,d) \geq \mathbf{1}      && \hspace{-.3cm} \text{ for } i=1, \cdots, L \nonumber \\
&  H_i(c_i,d+d_f-1) \leq -\gamma \cdot \mathbf{1}  && \hspace{-.3cm}\text{ for } i=1, \cdots, L \nonumber \\
& G_i(b_i,d) = F_i(c_i,d+d_f-1)   &&\hspace{-.3cm}\text{ for } i=1, \cdots, L \nonumber \\
&  \nonumber J_{i}(b_i,d,k) = J_{j}(b_j,d,l)  && \hspace{-.3cm} \text{ for } i,j=1, \cdots, L \text{ and } \\
&  && \qquad\quad k,l \in \Lambda_{i,j}(D_\Gamma)\label{eq:LP}
\end{align}
is positive, then the origin is an asymptotically stable equilibrium for System~\ref{eq:system}. Furthermore, \vspace{-0.05in}
\begin{small}
\begin{equation*}
V(x) = V_i(x) = \hspace{-0.05in} \sum_{\alpha\in E_{d,m_i}} \hspace{-0.05in} b_{i,\alpha} \prod_{j=1}^{m_i} (h^T_{i,j}x+g_{i,j})^{\alpha_j} \text{ for } x \in D_i, \, i=1, \cdots,L \vspace{-0.05in}
\end{equation*}
\end{small}
with $b_{i,\alpha}$ as the elements of $b_i$, is a piecewise polynomial Lyapunov function proving stability of System~\eqref{eq:system}.
\end{thm}

\vspace{0.1in}
\begin{proof}
 Let us choose \vspace{-0.05in}
\begin{small}
\begin{equation*}
V(x) = V_i(x) = \hspace{-0.05in} \sum_{\alpha\in E_{d,m_i}} \hspace{-0.05in} b_{i,\alpha} \prod_{j=1}^{m_i} (h^T_{i,j}x+g_{i,j})^{\alpha_j} \text{ for } x \in D_i, \, i=1, \cdots,L
\end{equation*}
\end{small}
In order to show that $V(x)$ is a Lyapunov function for system~\ref{eq:system}, we need to prove the following:
\begin{enumerate}
\item  $V_i(x)\geq x^Tx$ for all $x\in D_i,\;i=1,\cdots,L$,
\item  $D^+(V_i(x),f(x))\leq -\gamma \,x^Tx$ for all $x\in D_i,\;i=1,\cdots,L$ and for some $\gamma>0$,
\item $V(0) = 0$,
\item $V(x)$ is continuous on $\Gamma$.
\end{enumerate}
Then, by Theorem~\ref{thm:lyap}, it follows that System~\eqref{eq:system} is asymptotically stable at the origin. Now, let us prove items (1)-(4).
For some $d\in\mathbb{N}$, suppose $\gamma>0,\;b_i\text{ and }c_i\text{ for }i=1,\cdots,L$ is a solution to linear program \eqref{eq:LP}.

\noindent\textbf{Item 1.} First, we show that $V_i(x)\geq x^Tx$ for all $x\in D_i,\;i=1,\cdots,L$. From the definition of the D-decomposition in the theorem statement, $h_{i,j}^Tx+g_{i,j}\geq 0$,  for all $x\in D_i$, $j=1,\cdots,m_i$. Furthermore, $b_i\geq \mathbf{0}$. Thus,
\begin{equation}\label{eq:Vipos}
V_i(x) := \sum_{\alpha\in E_{d,m_i}} b_{i,\alpha} \prod_{j=1}^{m_i} (h^T_{i,j}x+g_{i,j})^{\alpha_j}\geq 0
\end{equation}
 for all $x\in D_i\backslash,\;i=1,\cdots,L$. From \eqref{eq:Hdef}, $H_i(b_i,d)\geq \mathbf{1}$ for each $i=1,\cdots,L$ implies that all the coefficients of the expansion of $x^Tx$ in $V_i(x)$ are greater than $1$ for $i=1,\cdots,L$. This, together with \eqref{eq:Vipos}, prove that $V_i(x)\geq x^Tx$ for all $x\in D_i,\;i=1,\cdots,L$.

\noindent\textbf{Item 2.} Next, we show that $D^+(V_i(x),f(x))\leq -\gamma x^Tx $ for all $x\in D_i,\;i=1,\cdots,L$. For $i=1,\cdots,L$, let us refer the elements of $c_i$ as $c_{i,\beta}$, where $\beta\in E_{d+d_f-1,m_i}$ . From \eqref{eq:LP}, $c_{i}\leq \mathbf{0}$ for $i=1,\cdots,L$. Furthermore, since  $h_{i,j}^Tx+g_{i,j}\geq 0$ for all $x\in D_i$, it follows that \vspace{-0.1in}
\begin{equation}\label{eq:Zineg}
Z_i(x) = \sum_{\beta \in E_{d+d_f-1}} c_{\beta,i} \prod_{j=1}^{m_i}  (h_{i,j}^Tx+g_{i,j})^{\beta_j}\leq 0 \vspace{-0.1in}
\end{equation}
 for all $x\in D_i,\;i=1,\cdots,L$. From \eqref{eq:Hdef}, $H_i(c_i,d+d_f-1)\leq -\gamma\cdot\mathbf{1}$ for $i=1,\cdots,L$ implies that all the coefficients of the expansion of $x^Tx$ in $Z_i(x)$ are less than $-\gamma$ for $i=1,\cdots,L$. This, together with \eqref{eq:Zineg}, prove that $Z_i(x)\leq -\gamma x^Tx$ for all $x\in D_i$, for $i=1,\cdots,L$. Lastly, by the definitions of the maps $G_i$ and $F_i$ in \eqref{eq:Gdef} and \eqref{eq:Fdef}, if $G_i(b_i,d) = F_i(c_i,d+d_f-1)$, then $\langle \nabla V_i(x),f(x)\rangle = Z_i(x)\leq -\gamma x^Tx$ for all $x\in D_i$ and $i\in\{1\cdots,L\}$.
 Since $D^+(V_i(x),f(x)) = \langle\nabla V_i(x),f(x)\rangle \text{ for all } x\in D_i$, it follows that $D^+(V_i(x),f(x))\leq -\gamma x^Tx $ for all $x\in D_i,\;i\in\{1\cdots,L\}$.

\noindent\textbf{Item 3.} Now, we show that $V(0) = 0$. By Lemma~\ref{Zero Origin}, $R_i(b_{i},d) = \mathbf{0}$ implies $V_i(0)=0$ for each $i\in\{1,\cdots,L\}$.

\noindent\textbf{Item 4.} Finally, we show that $V(x)$ is continuous for $x\in \Gamma$. By Lemma~\ref{Continuity Lemma}, $J_{i}(b_i,d,k) = J_{j}(b_j,d,l)$ for all $i,j\in\{1,\cdots,L\}$, $k,l\in\Lambda_{i,j}(D_\Gamma)$ implies that $V(x)$ is continuous for all $x\in\Gamma$.
\end{proof}

Using Theorem~\ref{thm:main_theorem}, we define Algorithm 1 to search for piecewise-polynomial Lyapunov functions to verify local stability of system \eqref{eq:system} on convex polytopes. We have provided a Matlab implementation for Algorithm 1 at: $\mathtt{www.sites.google.com/a/asu.edu/kamyar/Software}$. \vspace{-0.1in}

\begin{algorithm}

\begin{footnotesize}
\textbf{Inputs:}
\begin{itemize}
\item Vertices of the polytope: $p_i$ for $i=1, \cdots,K$
\item $h_{i,j}$ and $g_{i,j}$ for $i=1,\cdots,K$ and $j=1,\cdots,m_i$
\item Coefficients and degree of the polynomial vector field of ~\eqref{eq:system}
\item Maximum degree of the Lyapunov function: $d_{max}$
\end{itemize}

%

\While{ $d < d_{\text{max}}$}{	
  	\eIf{the LP defined in \eqref{eq:LP} is feasible}{		
  		Break the while loop \\		
 	}{
 	 	Set $d=d+1$\\	
 	}
}

\textbf{Outputs:}\\
\begin{itemize}
\item In case the LP in \eqref{eq:LP} is feasible then the output is the coefficients $b_{i,\alpha}$
of the Lyapunov function
\vspace{-0.1in}
\[
V(x)=V_i(x) = \hspace{-.2cm}\sum_{\alpha\in E_{d,m_i}} \hspace{-.1cm}b_{i,\alpha} \prod_{j=1}^{m_i} (h^T_{i,j}x+g_{i,j})^{\alpha_j} \text{ for } x \in D_i,\;i=1, \cdots,L
\]
\end{itemize}
\caption{Search for piecewise polynomial Lyapunov functions}
\end{footnotesize} \vspace{-0.1in}
\label{algorithm}
\end{algorithm}\vspace*{-0.2in}


\section{Complexity Analysis}
\label{sec:results}
In this section, we analyze and compare the complexity of the LP in \eqref{eq:LP} with the complexity of the SDPs associated with Polya's algorithm in~\cite{kamyar2013nonlinear} and an SOS approach using Positivstellensatz multipliers. For simplicity, we consider Lyapunov functions defined on a hypercube centered at the origin. Note that we make frequent use of the formula \vspace{-0.1in}
\[
N_{vars} := \sum_{i=0}^d \dfrac{(i+K-1)!}{i!(K-1)!}, \vspace{-0.1in}
\]
which gives the number of basis functions in $\Theta_d(\Gamma)$ for a convex polytope $\Gamma$ with $K$ facets.

\subsection{Complexity of the LP associated with Handelman's Representation}
We consider the following $D-$decomposition.
\begin{myass}
We perform the analysis on an $n-$dimensional hypercube, centered at the origin. The hypercube is decomposed into $L=2n$ sub-polytopes such that the $i$-th sub-polytope has $m=2n-1$ facets. Fig.~\ref{fig:polytope_complexity} shows the $1-$, $2-$ and $3-$dimensional decomposed hypercube.
\label{assum_decompose}
\end{myass}

\begin{figure}
\centering
\includegraphics[scale=1.6]{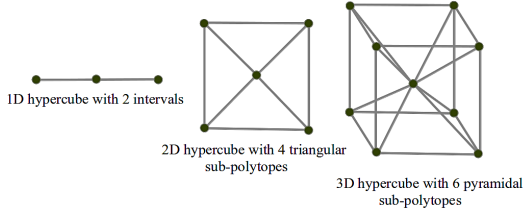} \vspace*{-0.15in}
\caption{Decomposition of the hypercube in $1-$,$2-$ and $3-$dimensions}
\label{fig:polytope_complexity} \vspace*{-0.3in}
\end{figure}

Let $n$ be the number of states in System~\eqref{eq:system}. Let $d_f$ be the degree of the polynomial vector field in System~\eqref{eq:system}. Suppose we use Algorithm 1 to search for a Lyapunov function of degree $d_V$. Then, the number of decision variables in the LP is \vspace{-0.1in}
\begin{small}
\begin{align}
N^H_{vars} \hspace*{-0.04in}  = L \left(  \sum_{d=0}^{d_V} \dfrac{(d+m-1)!}{d!(m-1)!} + \hspace*{-0.1in} \sum_{d=0}^{d_V+d_f-1} \hspace*{-0.05in} \dfrac{(d+m-1)!}{d!(m-1)!} -(d_V+1) \hspace*{-0.05in}  \right)
\label{eq:Nvar_H_general}
\end{align}
\end{small}
where the first term is the number of $b_{i,\alpha}$ coefficients, the second term is the number of $c_{i,\beta}$ coefficients and the third term is the dimension of $R_i(b_i,d)$ in \eqref{eq:LP}. By substituting for $L$ and $m$ in~\eqref{eq:Nvar_H_general}, from Assumption~\ref{assum_decompose} we have 
\[
N^H_{vars}=2n \hspace*{-0.03in} \left( \sum_{d=0}^{d_V} \dfrac{(d+2n-2)!}{d!(2n-2)!} + \hspace*{-0.1in} \sum_{d=0}^{d_V+d_f-1} \hspace*{-0.05in} \dfrac{(d+2n-2)!}{d!(2n-2)!}-d_V-1 \hspace*{-0.05in} \right) \hspace*{-0.03in}.
\]
Then, for large number of states, i.e., large $n$,
\[
N^H_{vars} \sim 2n \left( (2n-2)^{d_V} + (2n-2)^{d_V+d_f-1} \right) \sim n^{d_V+d_f}.
\]
Meanwhile, the number of constraints in the LP is
\begin{equation}
\label{eq:Ncons_H_general}
N^H_{cons} = N^H_{vars} + L \left( \sum_{d=0}^{d_V} \dfrac{(d+n-1)!}{d!(n-1)!} + \sum_{d=0}^{d_V+d_f-1} \dfrac{(d+n-1)!}{d!(n-1)!} \right),
\end{equation}
where the first term is the total number of inequality constraints associated with the positivity of $b_{i}$ and negativity of $c_{i}$, the second term is the number of equality constraints on the coefficients of the Lyapunov function required to ensure continuity ($J_{i}(b_i,d,k) = J_j(b_j,d,l)$ in the LP \eqref{eq:LP}) and the third term is the number of equality constraints associated with negativity of the Lie derivative of the Lyapunov function ($G_i(b_i,d) = F_i(c_i,d+d_f-1)$ in the LP~\eqref{eq:LP}).
By substituting for $L$ in~\eqref{eq:Ncons_H_general}, from Assumption~\ref{assum_decompose} for large $n$ we get
\[
N^H_{cons} \sim n^{d_V+d_f} + 2n(n^{d_V}+n^{d_V+d_f-1}) \sim n^{d_V+d_f}.
\]
The complexity of an LP using interior-point algorithms is approximately $O(N_{vars}^2 N_{cons})$~\cite{boydconvex}. Therefore the computational cost of solving the LP~\eqref{eq:LP} is
\[
\sim n^{3(d_V+d_f)}.
\]

\subsection{Complexity of the SDP associated with Polya's algorithm}
Before giving our analysis, we briefly review Polya's algorithm~\cite{kamyar2012hypercube} as applied to positivity of a polynomial on the hypercube. First, given a polynomial $T(x)$, for every variable $x_i \in [l_i,u_i]$, we define an auxiliary variable $y_i$ such that the pair $(x_i,y_i)$ lies on the simplex. Then, by using the procedure in~\cite{kamyar2012hypercube}, we construct a homogeneous version of $T$, defined as $\tilde{T}(x,y)$ so that $\tilde{T}(x,y)=T(x)$ for $(x_i,y_i)\in \Delta_i$. Finally, if for some $e \geq 0$ (Polya's exponent) the coefficients of $(x_1+y_1+ \cdots + x_n+y_n)^e\tilde{T}(x,y)$ are positive, then $T(x)$ is positive on the hypercube $[l_1,u_1] \times \cdots \times [l_n,u_n]$.

In~\cite{kamyar2013nonlinear}, we used this approach to construct Lyapunov functions defined on the hypercube. This algorithm used semidefinite programming to search for the coefficients of a matrix-valued polynomial $P(x)$ which defined a Lyapunov function as $V(x) = x^T P(x) x$. In~\cite{kamyar2013nonlinear}, we determined that the number of decision variables in the associated SDP was
\[
N^P_{vars} = \dfrac{n(n+1)}{2} \sum_{d=0}^{d_V-2}\dfrac{(d+n-1)!}{d! (n-1)!}.
\]
The number of constraints in the SDP was
\[
N^P_{cons} =  \dfrac{n(n+1)}{2} \left( (d_V+e-1)^n + (d_V+d_f+e-2)^n \right),
\]
where $e$ is Polya's exponent mentioned earlier. Then, for large $n$, 
$
N^P_{vars} \sim n^{d_V}$ and $N^P_{cons} \sim (d_V+d_f+e-2)^n.
$
Since solving an SDP with an interior-point algorithm typically requires $ O(N_{cons}^3+N_{var}^3 N_{cons} + N_{var}^2 N_{cons}^2 )$ operations~\cite{boydconvex}, the computational cost of solving the SDP associated with Polya's algorithm is estimated as
\[
\sim (d_V+d_f+e-2)^{3n}.
\]

\subsection{Complexity of the SDP associated with SOS algorithm}
To find a Lyapunov function for~\eqref{eq:system} over the polytope \vspace*{-0.05in}
\[
\Gamma= \left\{ x \in \mathbb{R}^n : w_i^Tx+u_i \geq 0, i\in \{1, \cdots, K \} \right\}
\]
using the SOS approach with Positivstellensatz multipliers~\cite{stengle}, we search for a polynomial $V(x)$ and SOS polynomials $s_i(x)$ and $t_i(x)$ such that for any $\epsilon > 0$ \vspace{-0.05in}
\[
V(x)- \epsilon x^Tx - \sum_{i=1}^K s_i(x) (w_i^Tx+u_i) \text{ is SOS} \quad \text{and} \vspace*{-0.2in}
\]
\[
- \langle \nabla V(x),f(x) \rangle - \epsilon x^T x -\sum_{i=1}^K t_i(x) (w_i^Tx+u_i) \text{ is SOS}. \vspace{-0.05in}
\]
Suppose we choose the degree of the $s_i(x)$ to be $d_V-2$ and the degree of the $t_i(x)$ to be $d_V+d_f-2$. Then, it can be shown that the total number of decision variables in the SDP associated with the SOS approach is
\begin{small}
\begin{equation}
N^S_{vars} = \dfrac{N_1(N_1+1)}{2} + K \dfrac{N_2(N_2+1)}{2} + K \dfrac{N_3(N_3+1)}{2},
\label{eq:NSvar}
\end{equation}
\end{small}
where $N_1$ is the number of monomials in a polynomial of degree $d_V/2$ , $N_2$ is the number of monomials in a polynomial of degree $(d_V-2)/2$ and $N_3$ is the number of monomials in a polynomial of degree $(d_V+d_f-2)/2$ calculated as \vspace*{-0.05in}
\begin{small}
\[
N_1= \hspace*{-0.05in} \sum_{d=1}^{d_V/2} \dfrac{(d+n-1)!}{(d)!(n-1)!}, \vspace*{-0.1in}
\]
\[
N_2= \hspace*{-0.05in} \sum_{d=0}^{(d_V-2)/2} \dfrac{(d+n-1)!}{(d)!(n-1)!} \quad \text{and} \quad
N_3= \hspace*{-0.1in} \sum_{d=0}^{(d_V+d_f-2)/2} \dfrac{(d+n-1)!}{(d)!(n-1)!}.
\]
\end{small}
The first terms in~\eqref{eq:NSvar} is the number of scalar decision variables associated with the polynomial $V(x)$. The second and third terms are the number of scalar variables in the polynomials $s_i$ and $t_i$, respectively. It can be shown that the number of constraints in the SDP is
\begin{equation}
\label{eq:NScons}
N^S_{cons} = N_1 + K \, N_2 + K \, N_3 + N_4,
\end{equation}
where \vspace{-0.1in}
\[
N_4 = \sum_{d=0}^{(d_V+d_f)/2} \dfrac{(d+n-1)!}{(d)!(n-1)!}.
\]
The first term in~\eqref{eq:NScons} is the number of constraints associated with  positivity of $V(x)$, the second and third terms are the number of constraints associated with positivity of the polynomials $s_i$ and $t_i$, respectively. The fourth term is the number of constraints associated with negativity of the Lie derivative. By substituting $K=2n$ (For the case of a hypercube), for large $n$ we have
$
N^S_{vars} \sim N_3^2 \sim n^{d_V+d_f-1} \text{ and}
$
\[
\quad N^S_{cons} \sim KN_3+N_4 \sim n \, N_3+N_4 \sim n^{ 0.5(d_V+d_f)}.
\]
Finally, using an interior-point algorithm with complexity $O(N_{cons}^3+N_{var}^3 N_{cons} + N_{var}^2 N_{cons}^2 )$ to solve the SDP associated the SOS algorithm requires
$
\sim n^{3.5(d_V+d_f)-3}
$
operations. As an additional comparison, we also considered the SOS algorithm for global stability analysis, which does not use Positivstellensatz multipliers. For a large number of states, we have
$
N^S_{vars} \sim n^{ 0.5d_V}  \quad \text{and} \quad N^S_{cons} \sim n^{ 0.5(d_V+d_f)}.
$
In this case, the complexity of the SDP is \vspace{-0.05in}
\[
\sim n^{1.5(d_V+d_f)}+ n^{2d_V+d_f}. \vspace{-0.05in}
\]

\subsection{Comparison of the Complexities}
We draw the following conclusions from our complexity analysis.

\noindent \textbf{1.}  For large number of states, the complexity of the LP \eqref{eq:LP} and the SDP associated with SOS are both \textbf{polynomial} in the number of states, whereas the complexity of the SDP associated with Polya's algorithm grows \textbf{exponentially} in the number of states. For a large number of states and large degree of the Lyapunov polynomial, the LP has the least computational complexity.

\noindent \textbf{2.} The complexity of the LP \eqref{eq:LP} scales linearly with the number of sub-polytopes $L$.

\noindent \textbf{3.} In Fig.~\ref{fig:complexity_comparison}, we show the number of decision variables and constraints for the LP and SDPs using different degrees of the Lyapunov function and different degrees of the vector field. The figure shows that in general, the SDP associated with Polya's algorithm has the least number of variables and the greatest number of constraints, whereas the SDP associated with SOS has the greatest number of variables and the least number of constraints.

\begin{figure}[ht]
\includegraphics[scale=0.23]{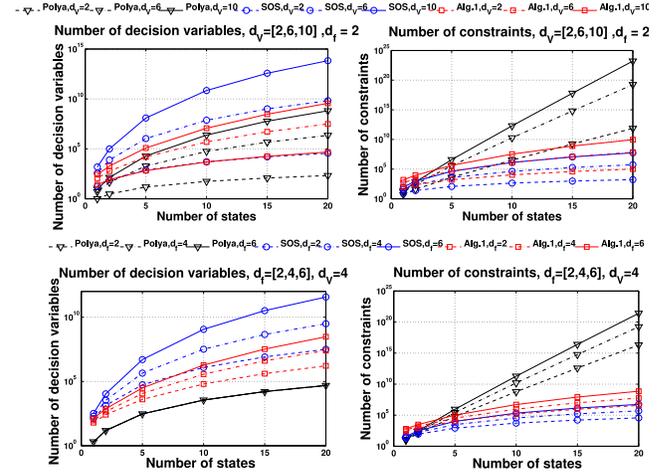}
\caption{Number of decision variables and constraints of the optimization problems associated with Algorithm 1, Polya's algorithm and SOS algorithm for different degrees of the Lyapunov function and the vector field $f(x)$}
\label{fig:complexity_comparison} \vspace*{-0.2in}
\end{figure}


\section{Numerical Results}
\label{sec:results}
In this section, we test the accuracy of our algorithm in approximating the region of attraction of a locally-stable nonlinear system known as the reverse-time Van Der Pol oscillator. The system is defined as\vspace*{-0.05in}
\begin{equation}
\label{eq:vanderpol}
\dot{x}_1 = -x_2, \; \dot{x}_2 =  x_1 + x_2 (x_1^2 -1). \vspace*{-0.05in}
\end{equation}
We considered the following convex polytopes:
\begin{enumerate}
\item Parallelogram $\Gamma_{P_s}$,
$P_s:=\{sp_i\}_{i=1,\cdots,4}$, where
\begin{small}
\[
\hspace{-0.12in} p_1= \begin{bmatrix}
-1.31 \\
0.18
\end{bmatrix},
p_2=\begin{bmatrix}
0.56\\
1.92
\end{bmatrix},
p_3=\begin{bmatrix}
-0.56\\
-1.92
\end{bmatrix},
p_4=\begin{bmatrix}
1.31\\
-0.18
\end{bmatrix} \vspace*{-0.05in}
\]
\end{small}
\item Square $\Gamma_{Q_s}$, $Q_s:=\{sq_i\}_{i=1,\cdots,4}$, where
\begin{small}
\[
q_1= \begin{bmatrix}
-1 \\
1
\end{bmatrix},
q_2=\begin{bmatrix}
1\\
1
\end{bmatrix},
q_3=\begin{bmatrix}
1\\
-1
\end{bmatrix},
q_4=\begin{bmatrix}
-1\\
-1
\end{bmatrix} \vspace*{-0.05in}
\]
\end{small}
\item Diamond $\Gamma_{R_s}$, $R_s:=\{sr_i\}_{i=1,\cdots,4}$, where
\begin{small}
\[
 r_1= \begin{bmatrix}
-1.41 \\
0
\end{bmatrix},
r_2=\begin{bmatrix}
0\\
1.41
\end{bmatrix},
r_3=\begin{bmatrix}
1.41\\
0
\end{bmatrix},
r_4=\begin{bmatrix}
0\\
-1.41
\end{bmatrix}
\]
\end{small}
\end{enumerate}
where $s\in\mathbb{R}_+$ is a scaling factor. We decompose the parallelogram and the diamond into 4 triangles and decompose the square into 4 squares. We solved the following optimization problem for Lyapunov functions of degree $d=2,4,6,8$: \vspace{-0.1in}
{\small
\begin{align*}
&  \max_{s \in \mathbb{R}^+} \; s \\
& \text{subject to} \quad \text{max }\gamma\text{ in LP \eqref{eq:LP} is positive, where}  \\
& \Gamma=\Gamma_{P_s} := \{ x \in \mathbb{R}^2 : x = \sum_{i=1}^4 \mu_i sp_i: \mu_i \geq 0 \text{ and } \sum_{i=1}^K \mu_i = 1\}.
\end{align*}
}
To solve this problem, we use a bisection search on $s$ in an outer-loop and an LP solver in the inner loop. Fig.~\ref{fig:quad_level_sets} illustrates the largest $\Gamma_{P_s}$, i.e. \vspace*{-0.05in}
\[
\Gamma_{P_{s^*}} := \{ x \in \mathbb{R}^n : x = \sum_{i=1}^4 \mu_is^* p_i: \mu_i\geq 0 \text{ and } \sum_{i=1}^4 \mu_i = 1 \} \vspace*{-0.05in}
\]
and the largest level-set of $V_i(x)$ inscribed in $\Gamma_{P_{s^*}}$, for different degrees of $V_i(x)$.
Similarly, we solved the same optimization problem replacing $\Gamma_{P_s}$ with the square $\Gamma_{Q_s}$ and diamond $\Gamma_{R_s}$. In all cases, increasing $d$ resulted in a larger maximum inscribed sub-level set of $V(x)$ (see Fig.~\ref{fig:diam_level_sets}). We obtained the best results using the parallelogram $\Gamma_{P_s}$ which achieved the scaling factor $s^*=1.639$.
The maximum scaling factor for $\Gamma_{Q_s}$ was $s^*=1.800$ and the maximum scaling factor for $\Gamma_{R_s}$ was $s^*=1.666$. \vspace{-0.15in}

\begin{figure}[htbp]
\centering
\includegraphics[scale=0.22]{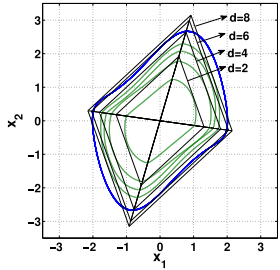} \vspace*{-0.1in}
\caption{Largest level sets of Lyapunov functions of different degrees and their associated parallelograms}
\label{fig:quad_level_sets} \vspace*{-0.3in}
\end{figure}

\begin{figure}[htbp]
\centering
\begin{subfigure}[Square polytopes]
{\includegraphics[width=0.54\columnwidth ]{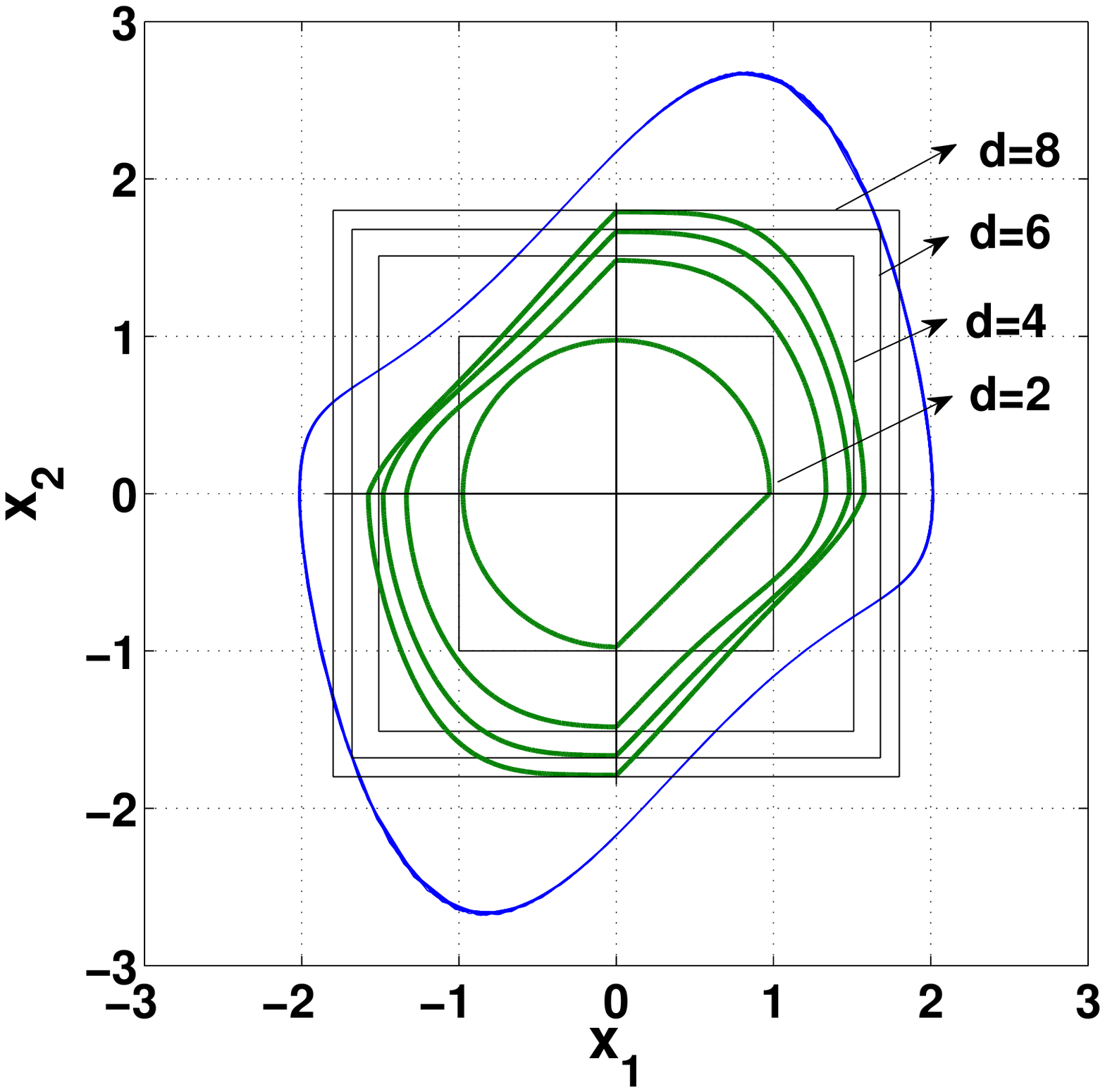}}
\end{subfigure}
~\hspace{-1.25cm}
\begin{subfigure}[{Diamond polytopes}]
{\includegraphics[width=0.54\columnwidth]{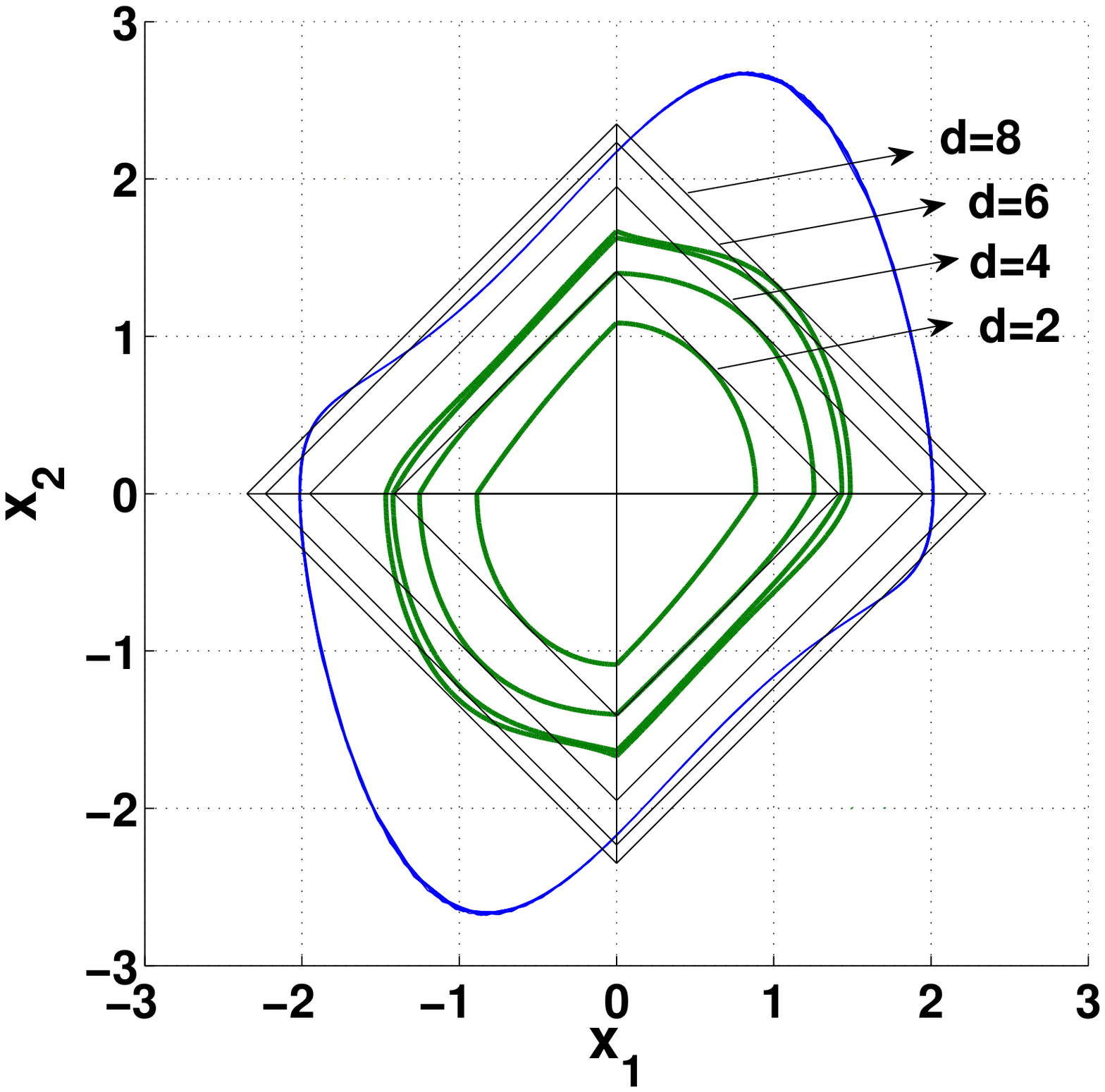}}
\end{subfigure}
 \vspace*{-0.15in}
 \caption{Largest level sets of  Lyapunov functions of different degrees and their associated polytopes}
\label{fig:diam_level_sets}  \vspace*{-0.1in}
\end{figure}


\section{Conclusion and future work}

In this paper, we propose an algorithm for stability analysis of nonlinear systems with polynomial vector fields.
The algorithm searches for piecewise polynomial Lyapunov functions defined on convex polytopes and represented in the Handelman basis. We show that the coefficients of the polynomial Lyapunov function can be obtained by solving a linear program. We also show that the resulting linear program has polynomial complexity in the number of states. We further improve the effectiveness of the algorithm by exploring the best polytopic domain for a given region of attraction. This work can also be potentially applied to stability analysis of switched systems and controller synthesis.


\section{Acknowledgements}
This material is based upon work supported by the Na-
tional Science Foundation under Grant Number 1301660.

\bibliographystyle{ieeetr}
\bibliography{Reza_Chai_CDC2014}

\begin{thebibliography}{10}

\bibitem{peet2012converse}
M.~M. Peet and A.~Papachristodoulou, ``A converse sum of squares {L}yapunov
  result with a degree bound,'' {\em IEEE Transactions on Automatic Control},
  vol.~57, no.~9, pp.~2281--2293, 2012.

\bibitem{Tarski}
A.~Tarski, ``A decision method for elementary algebra and geometry,'' {\em
  Random Corporation monograph, Berekley and Los Angeles}, 1951.

\bibitem{Artin}
E.~Artin, ``Uber die zerlegung definiter funktionen in quadra, quadrate,'' {\em
  Abh. Math. Sem. Univ. Hamburg}, vol.~5, pp.~85--99, 1927.

\bibitem{blum}
L.~Blum, {\em Complexity and real computation}.
\newblock Springer, 1998.

\bibitem{powers2011positive}
V.~Powers, ``Positive polynomials and sums of squares: Theory and practice,''
  {\em Real Algebraic Geometry}, p.~77, 2011.

\bibitem{papachristodoulou2013sostools}
A.~Papachristodoulou, J.~Anderson, G.~Valmorbida, S.~Prajna, P.~Seiler, and
  P.~Parrilo, ``{\uppercase{SOSTOOLS}} version 3.00 sum of squares optimization
  toolbox for matlab,'' {\em arXiv preprint arXiv:1310.4716}, 2013.

\bibitem{parrilo}
P.~A. Parrilo, {\em Structured semidefinite programs and semialgebraic geometry
  methods in robustness and optimization}.
\newblock PhD thesis, California Institute of Technology, 2000.

\bibitem{tan2008stability}
W.~Tan and A.~Packard, ``Stability region analysis using polynomial and
  composite polynomial lyapunov functions and sum-of-squares programming,''
  {\em IEEE Transactions on Automatic Control}, vol.~53, no.~2, pp.~565--570,
  2008.

\bibitem{prajna_hybrid}
S.~Prajna and A.~Papachristodoulou, ``Analysis of switched and hybrid
  systems-beyond piecewise quadratic methods,'' in {\em Proceedings of the 2003
  American Control Conference}, vol.~4, pp.~2779--2784, IEEE, 2003.

\bibitem{papachristodoulou2009analysis}
A.~Papachristodoulou, M.~M. Peet, and S.~Lall, ``Analysis of polynomial systems
  with time delays via the sum of squares decomposition.,'' {\em IEEE
  Transactions on Automatic Control}, vol.~54, no.~5, pp.~1058--1064, 2009.

\bibitem{polya_book}
G.~Hardy, J.~E. Littlewood, and G.~P{\'o}lya, {\em Inequalities}.
\newblock Cambridge University Press, 1934.

\bibitem{peres2007parameter}
R.~C. Oliveira and P.~L. Peres, ``Parameter-dependent lmis in robust analysis:
  characterization of homogeneous polynomially parameter-dependent solutions
  via lmi relaxations,'' {\em IEEE Transactions on Automatic Control}, vol.~52,
  no.~7, pp.~1334--1340, 2007.

\bibitem{kamyar2012hypercube}
R.~Kamyar and M.~Peet, ``Decentralized computation for robust stability of
  large-scale systems with parameters on the hypercube,'' in {\em Proceedings
  of the 2012 IEEE Conference on Decision and Control}, pp.~6259--6264, Dec
  2012.

\bibitem{castle2011polya}
M.~Castle, V.~Powers, and B.~Reznick, ``{P}{\'o}lya's theorem with zeros,''
  {\em Journal of Symbolic Computation}, vol.~46, no.~9, pp.~1039--1048, 2011.

\bibitem{de1996effective}
J.~A. de~Loera and F.~Santos, ``An effective version of {P}{\'o}lya's theorem
  on positive definite forms,'' {\em Journal of Pure and Applied Algebra},
  vol.~108, no.~3, pp.~231--240, 1996.

\bibitem{kamyar2012solving}
R.~Kamyar, M.~Peet, and Y.~Peet, ``Solving large-scale robust stability
  problems by exploiting the parallel structure of {P}olya's theorem,'' {\em
  IEEE Transactions on Automatic Control}, vol.~58, pp.~1931--1947, Aug 2013.

\bibitem{kamyar2013nonlinear}
R.~Kamyar and M.~M. Peet, ``Decentralized polya's algorithm for stability
  analysis of large-scale nonlinear systems,'' in {\em Proceedings of the 2013
  IEEE Conference on Decision and Control}, pp.~5858--5863, Dec 2013.

\bibitem{handelman}
D.~Handelman {\em et~al.}, ``Representing polynomials by positive linear
  functions on compact convex polyhedra,'' {\em Pac. J. Math}, vol.~132, no.~1,
  pp.~35--62, 1988.

\bibitem{2013handelmanlyapunov}
M.~A. Ben~Sassi, S.~Sankaranarayanan, X.~Chen, and E.~Abrah{\'a}m, ``Linear
  relaxations of polynomial positivity for polynomial lyapunov function
  synthesis,'' {\em preprint, arXiv:1407.2952}, 2014.

\bibitem{briat}
C.~Briat, ``Robust stability and stabilization of uncertain linear positive
  systems via integral linear constraints: {L}1-gain and {L}2-gain
  characterization,'' {\em International Journal of Robust and Nonlinear
  Control}, vol.~23, no.~17, pp.~1932--1954, 2013.

\bibitem{boydconvex}
S.~P. Boyd and L.~Vandenberghe, {\em Convex optimization}.
\newblock Cambridge university press, 2004.

\bibitem{stengle}
G.~Stengle, ``A nullstellensatz and a positivstellensatz in semialgebraic
  geometry,'' {\em Mathematische Annalen}, vol.~207, no.~2, pp.~87--97, 1974.

\end{thebibliography}
\end{document}